\documentclass[12pt]{article}
\usepackage{amssymb,amsmath,amsthm}
\usepackage{graphicx}
\usepackage{fullpage}

\newtheorem{theorem}{Theorem}
\newtheorem{lemma}[theorem]{Lemma}
\newtheorem{corollary}[theorem]{Corollary}
\newtheorem{proposition}[theorem]{Proposition}
\newtheorem{definition}[theorem]{Definition}

\title{Motions of a connected subgraph representing a swarm of robots inside a graph of work stations}
\author{Aar{\' o}n Atilano\footnotemark[3] \and Sebastian Bejos \footnotemark[3] \footnotemark[2] \and Christian Rubio-Montiel\footnotemark[3]}

\begin{document}
\maketitle

\def\thefootnote{\fnsymbol{footnote}}
\footnotetext[3]{Divisi{\' o}n de Matem{\' a}ticas e Ingenier{\' i}a, FES Acatl{\' a}n, Universidad Nacional Aut{\'o}noma de M{\' e}xico, Naucalpan, Mexico. {\tt aaron.atilano@comunidad.unam.mx, christian.rubio@acatlan.unam.mx}.}
\footnotetext[2]{Coordinaci{\' o}n de Ciencias Computacionales, Instituto Nacional de Astrof{\' i}sica, {\' O}ptica y Electr{\' o}nica, Puebla, Mexico. {\tt sebastian.bejos@inaoep.mx}.}

\begin{abstract} 
Imagine that a swarm of robots is given, these robots must communicate with each other, and they can do so if certain conditions are met. We say that the swarm is connected if there is at least one way to send a message between each pair of robots.  A robot can move from a work station to another only if the connectivity of the swarm is preserved in order to perform some tasks. We model the problem via graph theory, we study connected subgraphs and how to motion them inside a connected graph preserving the connectivity. We determine completely the group of movements.
\end{abstract}



%%%%%%%%%%%%%%%%%%%%%%%%%%%%%%%%%%%%%%%%%%%%%%%%%%%%%%%%%%%%%%%%%%%%%%%%%%%%%%%%%%%%%%%%

\section{Introduction}

In this work we model the following problem. Imagine that a swarm of robots is given, these robots must communicate with each other, and they can do so if certain conditions are met%, for example, if their mutual distance is smaller than some threshold or if they perceive each other’s images in their cameras, etcetera.
. We say that the swarm is connected if there is at least one way to send a message between each pair of robots. A message between robots can be sent if either there is a direct communication between them or if there are intermediate robots which can send the message. Some work stations in a region are also given, the number of work stations are at least the number of robots. A robot can move from one of these work stations to another only if the connectivity of the swarm is preserved. The swarm of robots has one fixed initial position and, in order to perform some tasks, the robots move from one station to another as needed, always maintaining the swarm connected. After a while the swarm of robots returns to its initial position. In order to achieve this goal it is not necessary that each robot returns to its initial position, we only care about the position of the whole swarm, so as long as each one of the initial positions are occupied and the swarm is connected, we say that it has returned to its original position.
Our intent in this paper is to study the different permutations that might appear once the swarm returns to its original position. In order to do so, we must also study the possible moves that the swarm can make, all moves must meet three conditions: 1. The connectivity of the swarm must be kept, 2. Only one robot is in each workstation at each time, 3. To avoid crashes, two robots are not allowed to swap positions.
%%Consider for instance one particular robot in the swarm and imagine for a moment that we mark this robot as red. The robot (and its swarm) moves around the work stations, performs some tasks and finally return to some position. Which are the possible positions that this robot might take as final once the swarm has returned? We study these permutations.

We model the problem using a graph as follows. The work stations are represented by the vertices of a graph, two vertices are connected by an edge if their corresponding workstations allow a couple of robots, one in each workstation, to communicate with each other. Notice that the initial positions of the robots induce a unique subgraph of our workstations graph and that every time a robot moves this induced 
subgraph might change. Since we are interested only in the moves that assure the connectivity of the swarm, both the workstations graph and every induced
subgraph must be connected. Under this model the subgraph of robots moves through the workstations graph and we ask how the permutations of the initial subgraph look like.

Related problems have been studied from a different perspective in the area of motion planning under the names of ``robots swarm'' and ``pebble motion'', for example in \cite{CKLL} and \cite{MR4036097}.

A classical related problem is the well-known ``15-puzzle'' which was generalized to graphs by Wilson \cite{MR0332555}, who proved that for any nonseparable graph, except for one, its associate group is the symmetric group unless the graph is bipartite, for which it is the alternating group.

While Wilson considered just an empty workstation, the problem was generalized to $k$ empty workstations in \cite{kornhauser1984coordinating} where it was also given a polynomial time algorithm that decides reachability of a target configuration. In \cite{MR1822278} optimal algorithms for specific graphs were explored. Colored versions were studied in \cite{MR2889522} and \cite{MR2672474}. In \cite{RW} it was proven that finding a shortest solution for the extended puzzle is NP-hard and therefore it is computationally intractable. 

In the following section we define formally the problem. In Section \ref{section3}, we prove that the set of possible movements is a group, and we define what we call the Wilson group (in honor of Richard M. Wilson). In Section \ref{section4}, we characterize such group when there are no ``empty workstations''. Finally, in Section \ref{section5}, we characterize the group in the case when there is at least one ``empty workstation''.

%%%%%%%%%%%%%%%%%%%%%%%%%%%%%%%%%%%%%%%%%%%%%%%%%%%%%%%%%%%%%%%%%%%%%%%%%%%%%%%%%%%%%%%%

\section{Definitions and basic results}\label{section2}

In this section we introduce definitions, terminology and basic results. All the graphs considered in the paper are finite and simple.

\begin{definition}
Let $G$ be a graph, $R$ a $k$-set and $f_t$ a function such that \[f_t\colon V(G)\rightarrow R\cup\{\emptyset\}.\]
We denote $f_t^{-1}(R)$ by $V_t$, and we say $f_t$ is an \emph{$R$-configuration over $G$} if $f_{t}\mid_{V_{t}}$ is bijective, where $t\in\Delta$ and $\Delta$ is a set of natural numbers.
\end{definition}

The elements of $R$ are called \emph{labels} and we use $R=[k]$ or a subset of $[k]$ for simplicity, where $[k]:=\{1,2,\dots,k\}$.

\begin{definition}
Let $G$ be a graph and $f_t$ a $[k]$-configuration over $G$. We say $f_t$ is a \emph{connected $[k]$-configuration over $G$} if the induced subgraph $G[V_t]$ of $G$ is connected.
\end{definition}

\begin{figure}[!htbp]
\begin{center}
\includegraphics{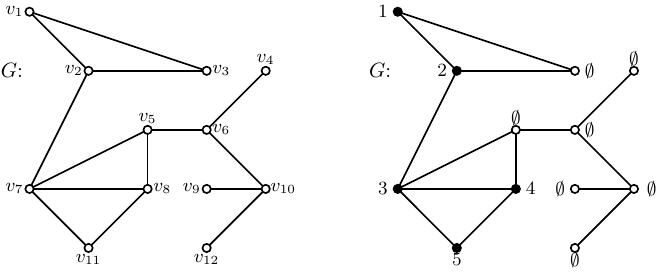}
\caption{(Left) A graph $G$. (Right) The graph $G$ labeled with a connected $[5]$-configuration $f_t$.}\label{fig01}
\end{center}
\end{figure}

If a vertex $v$ is such that $f_t(v)=\emptyset$, we say that it is \emph{empty}. The set of empty vertices is denoted by $V_\emptyset$. Figure \ref{fig01} shows an example of a connected $[k]$-configuration $f_t$ over a graph $G$, for $k=5$. We write \[f_t=\left(
	\begin{array}{@{\extracolsep{-2mm}}cccccccccccc}
	v_1&v_2&v_3&v_4&v_5&v_6&v_7&v_8&v_9&v_{10}&v_{11}&v_{12}\\
	1&2&\emptyset&\emptyset&\emptyset&\emptyset&3&4&\emptyset&\emptyset&5&\emptyset
	\end{array}\hspace{-2mm}\right).\]

Suppose that $f_t$ and $f_s$ are two connected $[k]$-configurations over a graph $G$. If $V_t=V_s$, for $f_t(v_j)=i$ and $f_s(v_j)=i'$ we have that $\sigma(i)=i'$ where $\sigma$ is a permutation of $V_t$.

\begin{definition}
Let $f_t$ and $f_s$ be two connected $[k]$-configurations over a graph $G$. If $V_t=V_s$ we say $f_t$ is \emph{similar} to $f_s$ and we denote it by $f_t \simeq f_s$.
\end{definition}

It is not hard to see that the relation $\simeq$ is an equivalence relation over the set of $[k]$-configurations over $G$. The equivalence class of $f_t$ is denoted by $[f_t]$. Therefore, a class $[f_t]$ is an unlabeled connected $[k]$-configuration $f_t$.

\subsection{Motioning connected subgraphs} 

In this subsection, we establish the rules to motion connected induced subgraphs preserving the connectivity.

\begin{definition}
Let $f_t$ be a $[k]$-connected configuration over a graph $G$. Let $w[f_t]$ be a function such that \[w[f_t]\colon V(G)\rightarrow \{0,1\}\]
for each $v\in V(G)$, $w[f_r](v)=1$ if $f_t(v)\in [k]$ and $w[f_r](v)=0$ otherwise.

The function $w[f_t]$ is called \emph{the weight function of $f_t$}.
\end{definition}

Clearly, if $f_t\simeq f_s$ then $w[f_t]=w[f_s]$. We recall that a cycle of a graph is denoted by $(v_1,v_2,\dots,v_r)$ where $v_1=v_r$. However, to keep our arguments as simple as possible, we choose to use $v_1\not = v_r$ and then $v_1$ is adjacent to $v_r$.

\begin{definition}
Let $f_t$ be a connected $[k]$-configuration over a graph $G$. An \emph{$r$-cycle} $p$ is a cycle $p=(v_1,v_2,\dots,v_r)$ such that $w[f_t](v_i)=1$, for all $i\in [r]$. And an \emph{$r$-path} $p$ is a path $p=(v_1,v_2,\dots,v_r)$ such that $w[f_t](v_i)=0$ if and only if $i=1$, that is, only the vertex $v_1$ has weight $0$.
\end{definition}

Now, we associate a permutation to an $r$-cycle or path. In this paper, the product of permutations means composition of functions on the left. For a detailed introduction on permutations we refer to the book of Rotman \cite{MR1307623}.

\begin{definition}
Let $f_t$ be a connected $[k]$-configuration over a graph $G$ and let $p$ be an $r$-cycle or a path $p=(v_1,v_2,\dots,v_r)$. An \emph{elementary $p$-movement} of $V_t$ is a permutation $\sigma_p$ such that
\[\sigma_p=(v_1v_2\dots v_r)=(v_1v_2)(v_2v_3)\dots(v_{r-1}v_r).\]  
\end{definition}

Hence, we can define configurations $f_s$ arising from a given configuration $f_t$.

\begin{definition}
Let $f_t$ be a connected $[k]$-configuration over a graph $G$ and $\sigma_p$ an elementary $p$-movement of $V_t$. The $[k]$-configuration $f_{t+1}=f_t\circ \sigma_p$ over $G$ is an \emph{elementary configuration movement} arising from $f_t$.
\end{definition}

Note that if $G$ is a tree such that $V(G)=V_t$ for a connected $[k]$-configuration then there is no elementary $p$-movements. And in general, it is possible that $G[V_{t+1}]$ is a disconnected subgraph.

\begin{definition}
Let $f_t$ be a connected $[k]$-configuration over a graph $G$ and $\sigma_p$ an elementary $p$-movement of $V_t$. If $f_{t+1}$ is an elementary configuration movement arising from $f_t$ such that it is a connected $[k]$-configuration then $f_{t+1}$ is called \emph{valid} as well as $\sigma_p$.
\end{definition}

Figure \ref{fig02} shows two valid elementary $p$-movements, namely, $\sigma_{p}=(v_5v_8v_7v_2v_1)$ and $\sigma_{p'}=(v_7v_8v_{11})$, therefore $f_{t+1}=f_t\circ \sigma_{p}$ and $f'_{t+1}=f_t\circ \sigma_{p'}$ where

\[f_{t+1}=\left(
	\begin{array}{@{\extracolsep{-2mm}}cccccccccccc}
	v_1&v_2&v_3&v_4&v_5&v_6&v_7&v_8&v_9&v_{10}&v_{11}&v_{12}\\
	\emptyset&1&\emptyset&\emptyset&4&\emptyset&2&3&\emptyset&\emptyset&5&\emptyset
	\end{array}\hspace{-2mm}\right)\text{ and }\]
\[f'_{t+1}=\left(
	\begin{array}{@{\extracolsep{-2mm}}cccccccccccc}
	v_1&v_2&v_3&v_4&v_5&v_6&v_7&v_8&v_9&v_{10}&v_{11}&v_{12}\\
	1&2&\emptyset&\emptyset&\emptyset&\emptyset&4&5&\emptyset&\emptyset&3&\emptyset
	\end{array}\hspace{-2mm}\right).\]

\begin{figure}[!htbp]
\begin{center}
\includegraphics{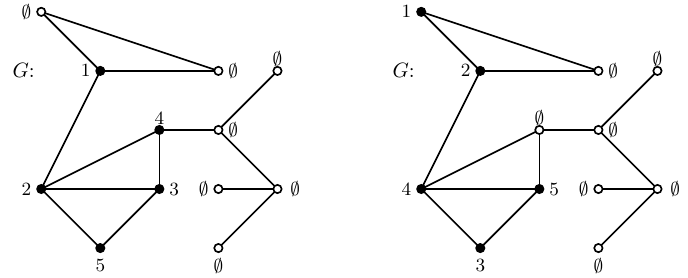}
\caption{(Left) The graph $G$ labeled with a connected $[5]$-configuration $f_{t+1}$. (Right) The graph $G$ labeled with a connected $[5]$-configuration $f'_{t+1}$.}\label{fig02}
\end{center}
\end{figure}

An example of an invalid elementary $p$-movement is $\sigma_{p''}=(v_3v_2v_7)$. 

Since we obtain a connected $[k]$-configuration $f_{t+1}$ from a connected $[k]$-configuration $f_{t}$ via a valid elementary $p$-movement, we have the following proposition.

\begin{proposition}\label{prop9}
A $\sigma_p$ is a valid elementary $p$-movement of $V_t$  if and only if $\sigma^{-1}_p$ is a valid elementary $p$-movement of $V_{t+1}.$
\end{proposition}

Consider the set of empty vertices $V_\emptyset$. For any permutation $\sigma$ in the symmetric group $S_{V_\emptyset}$ of $V_\emptyset$ we have that $f_t\circ \sigma=f_t$, therefore, $\sigma$ is a valid elementary $p$-movement.
Given a connected $[k]$-configuration $f_t$, we denote the set of valid elementary $p$-movements of $V_t$ as $\Gamma[f_t]$. Therefore $S_{V_\emptyset}\subseteq\Gamma[f_t]$.

\begin{proposition}
If $f_t\simeq f_s$ then $\Gamma[f_t]=\Gamma[f_s]$ and $f_t\circ \sigma\simeq f_s\circ \sigma$ for any $\sigma\in \Gamma[f_t]$.
\end{proposition}

By Proposition \ref{prop9} we have the following.

\begin{proposition}
$\sigma\in\Gamma[f_t]
%\setminus S_{V_\emptyset}
$ if and only if $\sigma^{-1}\in\Gamma[f_t\circ \sigma]$.
\end{proposition}

Next, we define a valid sequence of connected configurations.

\begin{definition}
Let $\{f_0,f_1,\dots,f_t\}$ be a set of connected $[k]$-configuration over a graph $G$. We say that it is a \emph{valid $f_0f_t$-sequence} if $f_s$ is a valid $[k]$-configuration arising from $f_{s-1}$ for all $s\in [t]$.
\end{definition}

Figure \ref{fig03} shows an example of a valid $f_0f_2$-sequence of a graph $H$ for 
\[f_0=\left(
	\begin{array}{@{\extracolsep{-2mm}}ccccccc}
	v_1&v_2&v_3&v_4&v_5&v_6&v_7\\
	\emptyset&1&\emptyset&2&3&\emptyset&\emptyset
	\end{array}\hspace{-2mm}\right).\]
Taking a $3$-path $p_1=(v_3,v_4,v_5)$ and the permutation $\sigma_{p_1}=(v_3v_4v_5)$ we get $f_1=f_0\circ \sigma_{p_1}$ obtaining $f_1=\left(
	\begin{array}{@{\extracolsep{-2mm}}ccccccc}
	v_1&v_2&v_3&v_4&v_5&v_6&v_7\\
	\emptyset&1&2&3&\emptyset&\emptyset&\emptyset
	\end{array}\hspace{-2mm}\right).$
Then, taking a $4$-path $p_1=(v_6,v_3,v_4,v_2)$ and the permutation $\sigma_{p_2}=(v_6v_3v_4v_2)$ we get $f_2=f_1\circ \sigma_{p_2}$ obtaining $f_2=\left(
	\begin{array}{@{\extracolsep{-2mm}}ccccccc}
	v_1&v_2&v_3&v_4&v_5&v_6&v_7\\
	\emptyset&\emptyset&3&1&\emptyset&2&\emptyset
	\end{array}\hspace{-2mm}\right).$

\begin{figure}[!htbp]
\begin{center}
\includegraphics{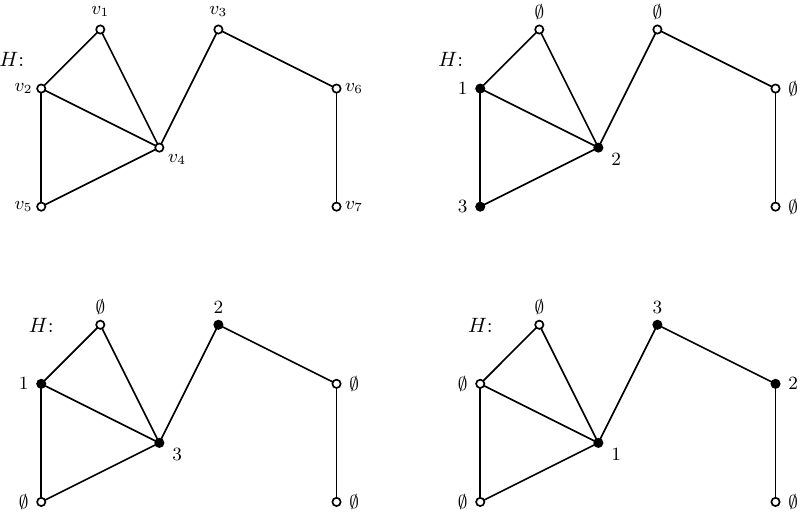}
\caption{(Left-top) A graph $H$. (Right-top) The graph $H$ labeled with a connected $[3]$-configuration $f_0$. (Left-bottom) The graph $H$ labeled with a connected $[3]$-configuration $f_1$. (Right-bottom) The graph $H$ labeled with a connected $[3]$-configuration $f_2$.}\label{fig03}
\end{center}
\end{figure}

Therefore, $f_2=f_0\circ \sigma_{p_1} \circ \sigma_{p_2}$, i.e., $f_2=f_0\circ \sigma$ where $\sigma=\sigma_{p_1} \circ \sigma_{p_2}$. In general, if $\{f_0,f_1,\dots,f_t\}$ is a set of $[k]$-configuration over a graph $G$, hence $f_i=f_{i-1}\circ \sigma_{p_i}$ for some $\sigma_{p_i}\in \Gamma[f_{i-1}]$ with $i\in [t]$ and then $f_t=f_0\circ\sigma$ for $\sigma=\sigma_{p_1} \circ \sigma_{p_2}\circ\dots\circ\sigma_{p_t}$.

\begin{definition}
Let $f_0$ and $f_t$ be two connected $[k]$-configurations over a graph $G$. If there exists a valid $f_0f_t$-sequence $\{f_0,f_1,\dots,f_t\}$ for which $f_i=f_{i-1}\circ \sigma_{p_i}$ for some $\sigma_{p_i}\in \Gamma[f_{i-1}]$ with $i\in [t]$, the permutation $\sigma=\sigma_{p_1} \circ \sigma_{p_2}\circ\dots\circ\sigma_{p_t}$ is called a \emph{valid movement} of $V_0$ having that $f_t=f_0\circ\sigma$.
\end{definition}

In a natural way we have the following two propositions.

\begin{proposition}\label{prop14}
A permutation $\sigma$ is a valid movement of $V_0$ if and only if $\sigma^{-1}$ is a valid movement of $V_{t}.$
\end{proposition}

\begin{proposition}\label{prop15}
If $f_t\simeq f_s$ then $f_t\circ \sigma\simeq f_s\circ \sigma$ for any valid movement $\sigma$ of $V_t$.
\end{proposition}

To end this section, we have the following theorem about the classes $[f_t]$.

\begin{theorem}\label{teo16}
Let $[f_t]$ and $[f_s]$ be two unlabeled connected $[k]$-configurations over a graph $G$. Then there exists a valid $f_tf_{t+r}$-sequence $\{f_t,f_{t+1},\dots,f_{t+r}\}$ for which $f_{t+r}\simeq f_s$.
\end{theorem}
\begin{proof}
Let $T_t$ and $T_s$ spanning trees be of $G[V_t]$ and $G[V_s]$, respectively. And let $P=(x_0,x_1,\dots,x_r)$ be a $T_tT_s$-geodesic. If the length $r$ of $P$ is positive, then $w[f_t(x_1)]=0$. Take a leaf $y$ of $T_t$ contained in $V_t\setminus \{x_0\}$ and let $p_1=(x_1,x_0,\dots,y',y)$ a path containing the $x_0y$-path of $T_t$. Therefore, $\sigma_{p_1}$ is a valid elementary $p_1$-movement of $V_t$. The tree $T_{t+1}$ with vertex set $V_{t+1}=V_t\cup\{x_1\}\setminus\{y\}$ and edge set $E(T_t)\cup\{x_1x_0\}\setminus\{yy'\}$ is a spanning tree of $G[V_{t+1}]$ with $P'=(x_1,\dots,x_r)$ a $T_{t+1}T_s$-geodesic shorter than $P$.

Now, we can assume that $r=0$. Consider a maximal component $T$ of $G[V_t\cap V_s]$ and take spanning trees $T_t$ and $T_s$ of $G[V_t]$ and $G[V_s]$, respectively, such that $T$ is a subgraph of them. If $T_t=T_s$ then $V_t=V_s$ and then $f_t\simeq f_s$. We can assume that there exist a leaf $y$ of $T_t$ contained in $V_t\setminus V(T)$ and then there exist a vertex $x$ in $V_s\setminus V(T)$ such that $xy''$ is an edge of $T_s$ and $y''$ is a vertex of $T$. Let $q_1=(x,y'',\dots,y',y)$ the path of $T_t\cup T_s$. Therefore, $\sigma_{q_1}$ is a valid elementary $q_1$-movement of $V_t$. The tree $T_{t+1}$ with vertex set $V_{t+1}=V_t\cup\{x\}\setminus\{y\}$ and edge set $E(T_t)\cup\{xy''\}\setminus\{yy'\}$ is a spanning tree of $G[V_{t+1}]$ and the maximal component $T'$ of $G[V_{t+1}\cap V_s]$ has order largest than $T$. Since the graph is finite, the result follows. 
\end{proof}

\begin{corollary} \label{cor17}
Let $[f_t]$ and $[f_s]$ be two unlabeled connected $[k]$-configurations over a graph $G$. Then there exists a valid movement $\sigma$ from $V_t$ to $V_s$.

\end{corollary}

\section{The Wilson group}\label{section3}

Given two connected $[k]$-configurations over a graph $G$, by Theorem \ref{teo16} and Corollary \ref{cor17}, we know that we can move the first one to the second one via connected subgraphs. In this section, we prove a similar result but considering the case when the labels are sorted.

Firstly, we define the following interesting set $\Phi$ regarding to valid movements. 
\begin{definition}
Let $f_t$ be a connected $[k]$-configuration over a graph $G$. The \emph{Wilson set} $\Phi[f_t]$ is the set of valid movements of $V_t$ such that $f_t=f_t\circ \sigma$ for all $\sigma\in\Phi[f_t]$.
\end{definition}

Clearly, the symmetric group $S_{V_\emptyset}$ of $V_\emptyset$ is a subset of $\Phi[f_t]$. The following proposition establishes that the Wilson set is independent to the labels of $f_t$.

\begin{proposition}
If $f_t\simeq f_s$ then $\Phi[f_t]=\Phi[f_s]$. In particular, if $\sigma \in \Phi[f_t]$ then $\Phi[f_t]=\Phi[f_t\circ\sigma]$.
\end{proposition}
\begin{proof}
Let $\sigma\in\Phi[f_t]$. Since $\sigma$ is a valid movement of $V_t$, by Proposition \ref{prop15}, we have $f_s\simeq f_t\simeq f_t\circ\sigma \simeq f_s\circ \sigma$, i.e., $\sigma\in\Phi[f_s]$ and then $\Phi[f_t]\subseteq\Phi[f_s]$. Analogously, $\Phi[f_s]\subseteq\Phi[f_t]$ and then $\Phi[f_t]=\Phi[f_s]$. Now, in particular, if $f_s=f_t\circ\sigma$ for some $\sigma\in \Phi[f_t]$ then $\Phi[f_t]=\Phi[f_t\circ\sigma]$.
\end{proof}

\begin{proposition}\label{proposition20}
Let $\sigma$ be a valid movement of $V_t$. If $\phi\in\Phi[f_t\circ \sigma]$ then $\sigma\circ\phi\circ\sigma^{-1}\in\Phi[f_t]$.
\end{proposition}
\begin{proof}
Since $\sigma$ is a valid movement of $V_t$, by Proposition \ref{prop14}, $\sigma^{-1}$ is a valid movement of $\sigma(V_t)$. On the other hand, since $\phi\in\Phi[f_t\circ \sigma]$ then $f_t\circ\sigma\simeq f_t\circ\sigma\circ\phi$. By Proposition \ref{prop15}, $\sigma^{-1}$ is a valid movement of $\sigma\circ\phi$ and then $f_t\simeq f_t\circ\sigma\circ\phi\circ\sigma^{-1}$. Because $\sigma\circ\phi\circ\sigma^{-1}$ is a valid movement of $V_t$ we have that $\sigma\circ\phi\circ\sigma^{-1}\in\Phi[f_t]$.
\end{proof}

Next, we prove that the Wilson set is, in fact, a group.

\begin{theorem}
The pair $(\Phi[f_t],\circ)$ is a group.
\end{theorem}
\begin{proof}
Let $\sigma_1,\sigma_2\in \Phi[f_t]$, then $f_t\simeq f_t\circ\sigma_1$ and $\Phi[f_t]=\Phi[f_t\circ\sigma_1]$ hence $\sigma_2\in \Phi[f_t\circ\sigma_1]$ and $f_t\circ\sigma_1\simeq f_t\circ\sigma_1\circ\sigma_2$. Then $f_t\simeq f_t\circ\sigma_1\circ\sigma_2$ and $\sigma_1\circ\sigma_2\in \Phi[f_t]$. Therefore, $\Phi[f_t]$ is closed under the operation $\circ$.

Now, it is clear that $f_t\simeq f_t\circ (1)$ and then the identity function is an element of $\Phi[f_t]$. Also it is clear the operation $\circ$ is associative.

Finally, let $\sigma\in \Phi[f_t]$. By Proposition \ref{prop14}, $\sigma^{-1}$ is a valid movement of $V_t$ and by Proposition \ref{prop15}, we have that $\sigma^{-1}\in \Phi[f_t]$.
\end{proof}

\begin{theorem}
Let $f_0$ and $f_t$ be connected $[k]$-configurations and $\sigma$ a valid movement of $V_0$ with $f_s=f_t\circ\sigma$ and $f_t\simeq f_s$. Then there exists a valid $f_0f_t$-sequence if and only if $f_t=f_s\circ\phi$ for some $\phi\in\Phi[f_t]$.  
\end{theorem}
\begin{proof}
Assume that there exists a valid $f_0f_t$-sequence and $\sigma_1$ is a valid movement from $V_0$ to $V_t$. By Proposition \ref{prop14}, $\sigma_1^{-1}$ is a valid movement from $V_t$ to $V_0$, i.e., $f_0=f_t\circ \sigma_1^{-1}$. Since $\sigma$ is a valid movement from $V_s$ to $V_t$ then $\sigma^{-1}$ is a valid movement from $V_t$ to $V_s$, i.e., $f_s=f_t\circ \sigma^{-1}$, then $f_t\circ\sigma^{-1}\simeq f_s\circ\sigma^{-1}=f_0$. Therefore, $f_0=f_t\circ \sigma_1^{-1}\simeq f_t\circ \sigma^{-1}$ and by Proposition \ref{prop15}, $f_t\simeq f_t\circ \sigma^{-1} \circ \sigma_1 $ and for $\phi=\sigma^{-1} \circ \sigma_1$ we have that $f_t=f_s\circ\phi$ with $\phi\in\Phi[f_t]$.

Now, we verify the converse, let $f_t=f_s\circ\phi$ be for some $\phi\in\Phi[f_t]$, then $\phi$ is a valid movement from $V_s$ to $V_t$ and $\sigma$ is a valid movement from $V_0$ to $V_s$, therefore $\sigma\circ\phi$ is a valid movement from $V_0$ to $V_t$ and the theorem follows.
\end{proof}

The existence of the valid movement $\sigma$ is guaranteed by Theorem \ref{teo16} and Corollary \ref{cor17}, hence, in order to verify the existences of a valid $f_0f_t$-sequence, we only need to find some $\phi\in\Phi[f_t]$ for which $f_t=f_s\circ\phi$.

\subsection{Some Wilson groups}

Therefore, we need to know the structure of the Wilson group of a given subgraph. We begin with some particular configurations.

\begin{theorem}\label{teo23}
Let $f_t$ be a connected $[n]$-configuration over a graph $G$ of order $n$.
\begin{enumerate}
\item If $G$ is a cycle, then $\Phi[f_t]$ is the cyclic group $\mathbb{Z}_n$.
\item If $G$ is a tree, then $\Phi[f_t]$ is the trivial group $\{(1)_{V}\}$.
\end{enumerate}
\end{theorem}
\begin{proof}
First, since $V_\emptyset$ is empty, then only the elementary valid movements are the cycles, for instance $\sigma=(v_1v_2\dots v_n)$ and then $\Phi[f_t]=\left\langle \sigma\right\rangle $ which is $\mathbb{Z}_n$.
Second, there is no elementary valid movements different to the identity permutation, and the result follows.
\end{proof}

\begin{theorem}\label{teo24}
Let $f_t$ be a connected $[k]$-configuration over a graph $G$ of order $n>k$. If $G$ is a cycle or a path then $\Phi[f_t]$ is $\{(1)_{V_t}\}\times S_{V_\emptyset}$.
\end{theorem}
\begin{proof}
The valid movements are given by $k$-paths only into the two opposite directions, namely $\sigma_{p_1}$ and $\sigma_{p_2}$, see Figure \ref{fig04}.
\begin{figure}[!htbp]
\begin{center}
\includegraphics{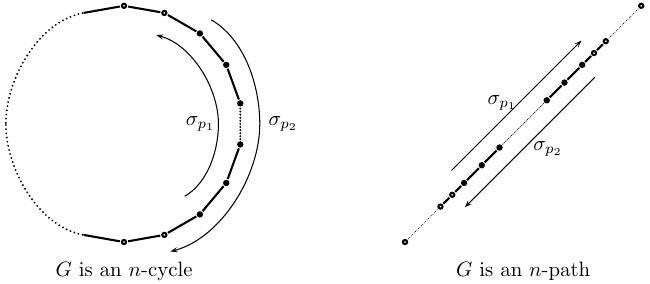}
\caption{The two possible direction of a $k$-path, for $k<n$, into a $n$-cycle or a $n$-path.}\label{fig04}
\end{center}
\end{figure}
Since the labels of $V_t$ are invariant under the valid movements, we have that $\Phi[f_t]=\{(1)_{V_t}\}\times S_{V_\emptyset}$ since any permutation of $V_\emptyset$ also leaves invariant the labels of $V_t$.
\end{proof}

\section{Saturated configurations}\label{section4}

In this section, we study the configurations without empty vertices, that is, each vertex has weight 1.

\begin{definition}
Let $G$ be a connected graph of order $n$. A connected $[n]$-configuration is called \emph{saturated}.
\end{definition}

Theorem \ref{teo23} states a result concerning to saturated configurations, namely, when $G$ is a cycle or a tree. Note that, the elementary movements are only given by cycles, that is, a vertex can be moved if it is in a cycle. 

\begin{corollary}
Let $f_t$ be a connected $[n]$-configuration over a unicyclic connected graph $G$ of order $n$ for which its cycle has order $k$. Then $\Phi[f_t]$ is the cyclic group $\mathbb{Z}_k$.
\end{corollary}

The following definitions are about edge-connectivity and edge-blocks.

\begin{definition}
A non-empty bridgeless connected subgraph $\textbf{B}$ of $G$ is called an \emph{edge-block} of $G$ if $\textbf{B}$ is maximal.
\end{definition}

Note that the Wilson group induces a (left) group action $\varphi$ into the set of vertices $\varphi\colon \Phi[f_t] \times V(G) \rightarrow V(G) $ where $\varphi(\sigma,v)=\sigma(v)$, therefore we have Theorem \ref{teo30}.

\begin{theorem}\label{teo30}
If $f_t$ is a saturated configuration over $G$ and $v\in V(\textbf{B})$ with $\textbf{B}$ an edge-block of $G$, then the orbit $\Phi[f_t]v$ of $v$ is $V(\textbf{B})$.
\end{theorem}
\begin{proof}
First, note that an edge-block could be separable if it contains cut-vertices. Now, let $u$ be a vertex of $V(\textbf{B})$. By Menger's Theorem, there exist two edge disjoint $uv$-paths such that they internally share only cut-vertices $v_1,\dots,v_{r-1}$. The union of this two paths is a union of cycles $p_1,\dots p_r$ where $v=v_0$ is a vertex of $p_1$ and $u=v_r$ is a vertex of $p_r$. The permutation $\sigma=\sigma^{a_r}_{p_r}\circ\dots\circ\sigma^{a_1}_{p_1}$ for which $\sigma^{a_i}_{p_i}(v_{i-1})=v_i$ for $i\in\{1,\dots,r\}$ maps $v$ to $u$, therefore $u\in \Phi[f_t]v$ and $V(\textbf{B})\subseteq \Phi[f_t]v$.
Finally, because a vertex $x$ out of $V(\textbf{B})$ is connected to $V(\textbf{B})$ via a bridge, it is not possible to move $x$ into $V(\textbf{B})$. Therefore $V(\textbf{B})=\Phi[f_t]v$.
\end{proof}

In consequence, the Wilson group of a saturated configuration of a graph in the product of the Wilson groups arising from each edge-block. Hence, we analyze the edge-blocks to know which is the Wilson group of a saturated graph.

We recall that $S_X$ denotes the symmetric group over $X$, while $A_X$ denotes the alternating group over $X$.

\begin{lemma}\label{lemma31}
Let $G=(V,E)$ be a graph which is a cycle $C$ with a chord with some subdivisions. And let $f_t$ be a saturated configuration of $G$. Then $\Phi[f_t]=S_V$.
\end{lemma}
\begin{proof}
The cycle $C$ can be interpreted as the union of the cycles $C'=(v_1\dots v_rw_m\dots w_1)$ and $C''=(w_1\dots w_mu_s\dots u_1)$ where the chord with its subdivisions is the path $T=(w_1\dots w_m)$ and the cycle $C$ is $(u_1\dots u_sw_mv_r\dots v_1w_1)$, see Figure \ref{fig05}.
\begin{figure}[!htbp]
\begin{center}
\includegraphics{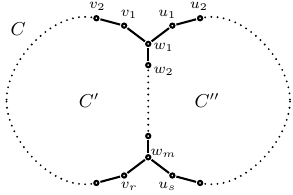}
\caption{Cycles of Lemma \ref{lemma31}.}\label{fig05}
\end{center}
\end{figure}

Since for the permutations $\sigma_{C'}$ and $\sigma_{C''}$ is turning clockwise while for the permutations $\sigma_C$ is turning counterclockwise, we have that
\[\sigma_{C}\circ\sigma_{C''}\circ\sigma_{C'}=(v_1w_1).\]
We call $\sigma'$ to this transposition. Now, we show $(vu)\in \Phi[f_t]$ for any $v,u\in V(G)$ with $v\not=u$. Without lossing of generality, $v=w_1$ and $u\in V(C'')$. For some $i$, $\sigma^i_{C''}$ we have $\sigma^i_{C''}(u)=v$. Suppose $u\not\in\{w_2,\dots,w_m\}$, then the permutation $\sigma=\sigma^i_{C''}\circ\sigma^{-1}_{C'}$ is such that $\sigma(v)=u$ with $v\sim u$. If $u\in\{w_2,\dots,w_m\}$, then $\sigma=\sigma^{i+1}_{C''}\circ\sigma^{-1}_{C'}$ is such that $\sigma(v)=u$ such that $v\sim u$. Hence, $(vu)=\sigma^{-1}\circ \sigma'\circ \sigma$ and $(vu)\in \Phi[f_t]$ and then $\Phi[f_t]=S_V$.
\end{proof}

A permutation is an element of $A_X$ if and only if it is a product of an even number of transpositions in $X$. Since every 3-cycle $(ijk)$ is the product of two transpositions $(ij)(ik)$ and the product of two transpositions $(ij)(kl)$ is the product of two 3-cycles $(ikj)(kjl)$ then the alternating group is generated by 3-cycles.

\begin{lemma}\label{lemma32}
Let $G=(V,E)$ be a graph which is the union of two cycles $C$ and $C'$ with exactly a cut vertex. And let $f_t$ be a saturated configuration of $G$. Then $\Phi[f_t]=S_V$ if $C$ or $C'$ is even, otherwise $\Phi[f_t]=A_V$.
\end{lemma}
\begin{proof}
First, we prove $A_V\subseteq \Phi[f_t]$, that is $(wvu)$ is an element of $\Phi[f_t]$ for any $w,v,u\in V$. We can assume $w=w_1$ since, for some $j$, $\sigma=\sigma^j_{C}$ or $\sigma=\sigma^j_{C'}$ is such that $\sigma(w)=w_1$ and then $(wvu)=\sigma^{-1}\circ(w_1vu)\circ\sigma$. 

Second, let $C=(w_1v_1\dots v_r)$ be and $C'=(w_1u_1\dots u_s)$, see Figure \ref{fig06}. Therefore $\sigma^{-1}_{C'}\circ\sigma^{-1}_{C}\circ\sigma_{C'}\circ\sigma_C=(w_1v_1u_1)$. We call $\sigma'$ to this 3-cycle. 

Then, we divide the proof into three cases:
\begin{figure}[!htbp]
\begin{center}
\includegraphics{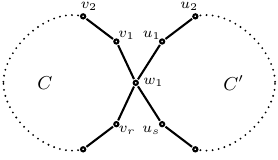}
\caption{Cycles of Lemma \ref{lemma32}.}\label{fig06}
\end{center}
\end{figure}
\begin{enumerate}
\item If $v=v_1$ and $u\in V(C')$, for some $j$, $\sigma=\sigma^j_{C'}$ is such that $\sigma(u)=w_1$ then $(w_1vu)=\sigma^{-1}\circ\sigma^{-1}_C\circ\sigma\circ\sigma_C$.
\item If $v=v_1$ and $u\in V(C)$, for some $j$, $\sigma=\sigma^j_{C}$ is such that $\sigma(u)=w_1$ then $(w_1vu)=\sigma^{-1}\circ\sigma_{C'}\circ\sigma'\circ\sigma^{-1}_{C'}\circ\sigma$.
\item Similar to (1.), but $v\in V(C)\setminus\{v_1\}$. If $u=u_1$ then the case is analogous to (1.) by symmetry. If $u\not=u_1$ then $\sigma=\sigma^{-1}_{C'}\circ\sigma^{j}\circ\sigma_{C'}$, for $\sigma^j(v)=v_1$ and some $j$, gets a configuration similar to the case (1.)
\end{enumerate}
Now, if $C$ and $C'$ are odd cycles, then $\sigma_C$ and $\sigma_{C'}$ are even permutations, therefore $A_V= \Phi[f_t]$. On the other hand, if $C$ or $C'$ is even, then it is an odd permutations, therefore $S_V= \Phi[f_t]$.
\end{proof}

In order to prove our main results, we define the following concept.

\begin{definition}
An edge-block of a graph is called \emph{weak} if it is a cycle or if every two cycles sharing vertices have exactly a vertex in common.
\end{definition}

\begin{figure}[!htbp]
\begin{center}
\includegraphics{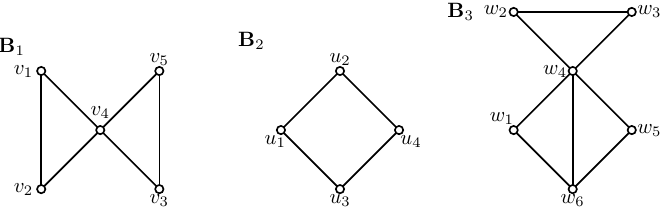}
\caption{The edge-blocks $\textbf{B}_1$ and $\textbf{B}_2$ are weak but $\textbf{B}_3$ is not.}\label{fig07}
\end{center}
\end{figure}

\begin{lemma}\label{lemma34}
Let $\textbf{B}$ be an edge-block that is not a cycle and $f_t$ be a saturated configuration. If $\textbf{B}$ is weak such that every cycle is odd, then $\Phi[f_t]=A_V$ otherwise $\Phi[f_t]=S_V$.
\end{lemma}
\begin{proof}
Let $\textbf{B}$ be as defined above and suppose that it is weak and every cycle is odd. We prove that $(uvw)\in \Phi[f_t]$ for any $u,v,w\in V(\textbf{B})$ as follows: since $u,v,w$ are in the same orbit and there are (at least) two cycles $C$ and $C'$ with exactly a cut vertex in common. We can send there all of them via a permutation $\sigma$, by Lemma \ref{lemma32}, we can get this 3-cycle there and then, via $\sigma^{-1}$ we get the desire 3-cycle. Because the parity of the odd cycles, $\Phi[f_t]=A_V$.

On the other hand, if $\textbf{B}$ contains an even cycle $\Phi[f_t]=S_V$ or if it is not weak, we can assume that $C$ and $C'$ have in common a path with more than a vertex therefore any transposition $(uv)$ can be done via a permutation $\sigma$ sending $u$ and $v$ to the cycles $C$ and $C'$, by Lemma \ref{lemma31}, we can get this 2-cycle there and finally, via $\sigma^{-1}$ we get the desire transposition getting that $\Phi[f_t]=S_V$.
\end{proof}

Now, we can describe the Wilson group of a saturated configuration.

\begin{theorem}
Let $f_t$ be a saturated configuration of a graph $G$, then 
\[\Phi[f_t]=\stackrel{r}{\underset{i=1}{\prod}}\varGamma_{i}\]
where $\Gamma_i$ is a cyclic group, an alternating group or a symmetric group.
\end{theorem}
\begin{proof}
Let $\textbf{B}_i$ be the set of edge-blocks $G$, for $i\in\{1,\dots,r\}$. By Theorem \ref{teo30}, the non-trivial orbits are $V(\textbf{B}_i)$ for each $i\in\{1,\dots,r\}$. Hence, for the set $V(\textbf{B}_i)$, its Wilson group is cyclic, by Theorem \ref{teo23}, or it is an alternating group or a symmetric group by Lemma \ref{lemma34} and the result follows.
\end{proof}

\section{No-saturated configurations}\label{section5}

In this section, we only study the Wilson group for no-saturated configurations. The main difference between saturated and no-saturated configurations is the existence of valid movements given by paths. Theorem \ref{teo24} is an example of this fact.

To begin with, we analyze the behavior of the bipartite complete graph $K_{1,3}$, also called as the 3-star graph, which is a relevant graph in no-saturated configurations.

Let $G$ be a graph containing at least a 3-star subgraph and $f_t$ a connected configuration over $G$ such that

\[f_t=\left(
	\begin{array}{@{\extracolsep{-2mm}}cccccc}
	\dots&v&v_1&u_1&w_1&\dots\\
	\dots&1&2&3&4&\dots
	\end{array}\hspace{-2mm}\right)\]
where $v$ is a vertex of degree at least 3 and $v_1$, $u_1$ and $w_1$ are adjacent to $v$. Figure \ref{fig08} shows the sequence of movements to generate the transposition $(vv_1)$ supposing some empty vertices.
\begin{figure}[!htbp]
\begin{center}
\includegraphics{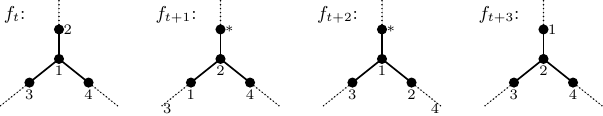}
\caption{Movements over a 3-star.}\label{fig08}
\end{center}
\end{figure}
Before to verify the details to produce the movements to generate such transposition, we give the following definition.
\begin{definition}
Let $G$ be a graph and $uv$ an edge of $G$. A vertex $w$ is \emph{in the direction of $u$ with respect to $v$} if there exists a $wv$-path containing $u$.
\end{definition}

Therefore, the set of vertices of the component of $G-v$ containing $u$ are the vertices in the direction of $u$ with respect to $v$. If $v$ is not a cut-vertex, each vertex is in the direction of $u$ with respect to $v$, for every $u$ in $N(v)$. The set of empty vertices in the direction of $u$ with respect to $v$ is denote by $B_v[f_t](u)$ and its cardinality by $b_v[f_t](u)$ (or simply by $B_v(u)$ and $b_v(u)$, respectively, when $f_t$ is understood), see Figure \ref{fig09}.

\begin{figure}[!htbp]
\begin{center}
\includegraphics{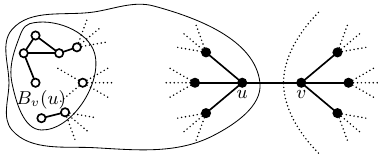}
\caption{A set $B_v(u)$.}\label{fig09}
\end{center}
\end{figure}

\begin{proposition}
Let $G$ be a connected graph. If $G[B_v(u)]$ is a tree, then there exists a spanning tree of $G$ containing $G[B_v(u)]$.
\end{proposition}

\begin{lemma}\label{lemma38}
Let $f_t$ be a connected $[k]$-configuration over a graph $G$ and $(w,v,u)$ a path such that $u,v\in V_t$. If $b_v(w)>0$ then there exists a cycle or a path $p$ containing $(w,v,u)$ for which $\sigma_p$ is in $\Gamma [f_t]$.
\end{lemma}
\begin{proof}
Since $b_v(w)>0$, there exists an empty vertex in the direction of $w$ with respect to $v$. Therefore, there is an $r$-path $p_1=(w_r,\dots,w_1=w,v)$ with $w_r$ an empty vertex.

On one hand, if $u=w_i$ for some $i\in [r-1]$ then for the cycle $p=(w_1=w,v,u=w_i,w_{i-1},\dots,w_2)$, $\sigma_p$ is in $\Gamma [f_t]$.

On the other hand, if $u\not =w_i$ for all $i\in [r-1]$ and let $T$ a spanning tree of $G[V_t]$ containing the path $p=(w_{r-1},\dots,w_1=w,v,u=u_0,\dots,u_s)$ where $u_s$ is a leaf of $T$, and then $\sigma_p$ is in $\Gamma [f_t]$, see Figure \ref{fig10}. 

\begin{figure}[!htbp]
\begin{center}
\includegraphics{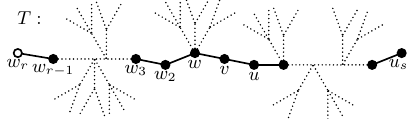}
\caption{The spanning tree in the proof of Lemma \ref{lemma38}.}\label{fig10}
\end{center}
\end{figure}
\end{proof}

\begin{theorem}\label{teo39}
Let $K_{1,3}$ be a star subgraph of $G$ with partition $(\{v\},\{u_1,v_1,w_1\})$ such that $b_v(u_1),b_v(w_1)>0$. If $v,v_1\in V_t$ and $vv_1$ is a bridge, then $(vv_1)\in \Phi[f_t]$. 
\end{theorem}
\begin{proof}
Let $f_t$ be a connected $[k]$-configuration over $G$. For the paths $(u_1,v,v_1)$ and $(w_1,v,v_1)$, by Lemma \ref{lemma38}, there exist the permutations
\[\sigma_1=(u_s\dots u_{1}vv_1\dots v_r)\textrm{ and }\sigma_1'=(w_m\dots w_1vv_1\dots v_r).\]
Since $vv_1$ is a bridge, we can assume that the corresponding paths of the permutations are sharing the leaf $v_r$.

If $u_s=w_m$ then the vertex $v$ is in a cycle $(v,u_1,\dots,u_i=w_j,w_1)$, see Figure \ref{fig11}. Let $\sigma_2$ and $\sigma_3$  be the permutations
\[\sigma_2=(vu_1\dots u_i\dots w_1)\textrm{ and }\sigma_3=(v_r\dots v_1vw_1\dots u_i\dots u_s).\]
Then $\sigma=\sigma_3\circ\sigma_2\circ\sigma_1=(vv_1)\in\Phi[f_t]$.

\begin{figure}[!htbp]
\begin{center}
\includegraphics{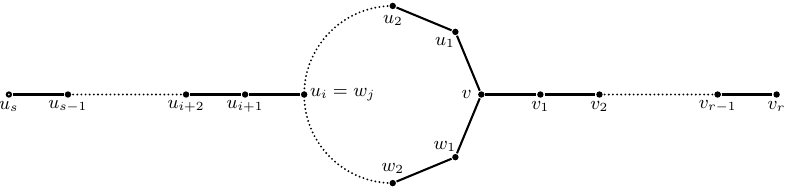}
\caption{Two paths sharing both leaves in the proof of Theorem \ref{teo39}.}\label{fig11}
\end{center}
\end{figure}

If $u_s\not=w_m$, consider the following permutations, see Figure \ref{fig12}.
\[\sigma_5=(w_m\dots w_1vu_1\dots u_s) \textrm{ and } \sigma_6=(v_r\dots v_1vw_1\dots w_m) \textrm{ and } \sigma_7=(u_sw_s).\]
Therefore $\sigma=\sigma_7\circ\sigma_6\circ\sigma_5\circ\sigma_1 \in \Phi[f_t]$ because $\sigma_7\in \Phi[f_t]$.
\begin{figure}[!htbp]
\begin{center}
\includegraphics{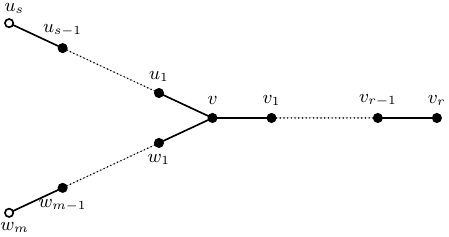}
\caption{Two paths sharing only a leaf in the proof of Theorem \ref{teo39}.}\label{fig12}
\end{center}
\end{figure}
\end{proof}

In order to analyze the Wilson group and use the previous ideas, we give the following definition.

\begin{definition}
Let $x,y\in V_t$ be and $v$ a vertex of degree 3 or more. We say that $v$ is an exchange-vertex for the pair $\{x,y\}$ if there is a vertex $v_1$ adjacent to $v$ and a valid permutation $\sigma$ such that $\sigma(v)=x$, $\sigma(v_1)=y$ and $f_t\circ\sigma$ satisfy the hypothesis of Theorem \ref{teo39}.

Moreover, if $x\sim y$ we say $v$ is an exchange-vertex for the edge $xy$.
\end{definition}

Note that if $v$ is an exchange-vertex of $\{x,y\}$ then $(xy)\in\Phi[f_t]$ because Proposition \ref{proposition20}, indeed, if $(vv_1)\in\Phi[f_t\circ\sigma]$ then $(xy)=\sigma\circ(vv_1)\circ\sigma^{-1}\in\Phi[f_t]$.

\begin{lemma}
If there exists an exchange-vertex for the pair $\{x,y\}$ then $(xy)\in\Phi[f_t]$.
\end{lemma}

\begin{theorem}\label{teo40}
If $(u_1,u_2,\dots,u_m)$ is a path such that for each edge there is an exchange vertex, then $(u_1u_m)\in\Phi[f_t]$.
\end{theorem}
\begin{proof}
Since each transposition $(u_iu_{i+1})\in\Phi[f_t]$ with $i\in [m-1]$ and therefore \[(u_1u_2)\circ(u_2u_3)\circ\dots\circ (u_{m-1}u_m)\circ(u_{m-1}u_{m-2})\circ\dots\circ (u_2u_1)=(u_1u_m)\in\Phi[f_t] .\]
\end{proof}

Now, we analyze the edge-blocks of a graph $G$ when it is not a cycle and $f_t$ is a no-saturated configuration.

\begin{lemma}\label{lemma43}
If $\textbf{B}$ is an edge-block of $G[V_t]$ for which every vertex has weight 1, then for each edge $xy\in E(\textbf{B})$ there exists an exchange-vertex for $xy$.
\end{lemma}
\begin{proof}
Since $f_t$ is a no-saturated configuration there is a vertex $v_1\not\in V(\textbf{B})$ adjacent to some vertex $v\in V(\textbf{B})$ such that $b_v(v_1)>0$. Since $\textbf{B}$ is a bridgeless subgraph, let $u,w$  adjacent to $v$ in $\textbf{B}$ living in a cycle there. Note that $v$ has degree at least 3.

Let $xy$ be an edge of $\textbf{B}$. Without loss of generality, we can assume that $(xu)\circ(yv)\in\Phi[f_y]$ by Theorem \ref{teo23} and Lemma \ref{lemma34}. Let $f_s=f_t\circ(xu)\circ(yv)$.

Note that the path $(v_1,v,u)$ satisfies the hypothesis of Lemma \ref{lemma38}. Hence, there exists a cycle or a path $p$ containing $(v_1,v,u)$ for which $\sigma_p\in\Gamma[f_s]$ with $\sigma_p(v_1)=v$ and $\sigma_p(v)=u$. Moreover, if $vv_1$ is not a bridge, we can delete edges adjacent to $v$ but not in $E(\textbf{B})$, $b_v[f_s\circ\sigma_p](w)>0$ and $b_v[f_s\circ\sigma_p](u)>0$. By Theorem \ref{teo39}, the vertex $v$ is an exchange-vertex for $xy$.

\begin{figure}[!htbp]
\begin{center}
\includegraphics{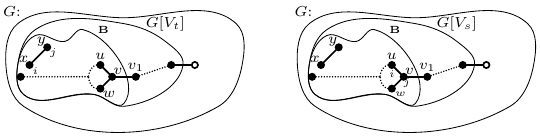}
\caption{The vertex $v$ is an exchange-vertex of $xy$.}\label{fig13}
\end{center}
\end{figure}
\end{proof}

The previous result can be established for vertex-blocks instead of edge-blocks because if $\textbf{B}$ is a vertex-block of $G[V_t]$ for which every vertex has weight 1, consider an edge $vv_1\in E(\textbf{B})$. If $vv_1$ is not a bridge, $vv_1$ is in an edge-block and there exists an exchange-vertex for $vv_1$. And if $vv_1$ is a bridge following the hypothesis of Theorem \ref{teo39}, then $v$ its an exchange-vertex for $vv_1$.

\begin{theorem}\label{teo44}
Let $xy$ be an edge for which its vertices have weight 1. If $xy$ is not a bridge an $G$ is not a cycle, then there exists an exchange-vertex for $xy$.
\end{theorem}
\begin{proof}
If $xy$ belongs to an edge-block of $G[V_t]$, by Lemma \ref{lemma43}, the result is done. Now assume to the contrary. By hypothesis, $xy$ belongs to a cycle $C=(u_1=x,u_2=y,\dots,u_s)$ of $G$ and $u_i$ has degree at least 3 for some $i\in [s]$.

Via a valid movement $\sigma$, we can move the edge $xy$ to the edge $u_iu_{i-1}$, see Figure \ref{fig14}, and by Theorem \ref{teo39}, $u_i=\sigma(y)$ (or $u_i=\sigma(x)$) is an exchange-vertex of $\sigma(x)\sigma(y)$, then $u_i$ is an exchange-vertex of $xy$.

\begin{figure}[!htbp]
\begin{center}
\includegraphics{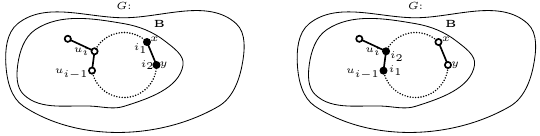}
\caption{The vertex $v$ is an exchange-vertex of $xy$.}\label{fig14}
\end{center}
\end{figure}
\end{proof}

\begin{theorem}\label{teo43}
If $u$ and $v$ are two different vertices of an edge-block $\textbf{B}$ of a graph $G$, both of weight 1 and $G$ is not a cycle, then $(uv)\in\Phi[f_t]$
\end{theorem}
\begin{proof}
Since there exist a path $P$ contained in $\textbf{B}$ for which every vertex is in $G[V_t]$, by Lemma \ref{lemma43} and Theorem \ref{teo44} the result follows.
\end{proof}

\begin{corollary}
If $G$ is an edge-block, $G$ is not a cycle and $f_t$ is a no-saturated configuration, then $\Phi[f_t]=S_{V_t}$.
\end{corollary}

Before to analyzing the bridges of a graph with no-saturated configurations, consider the set of vertices $C_v$ for which $v$ is an exchange-vertex, i.e.,
\[C_v=\{x,y\in V_t\colon v\textrm{ is an exchange-vertex for the pair } \{x,y\} \}.\]

In order to see the relation between $C_v$ and the orbits of $\Phi[f_t]$ we have the following definition and results.

\begin{corollary}
Let $\textbf{B}$ be an edge-block of $G$. If $v$ is an exchange-vertex for some edge of $\textbf{B}$, then $v$ is an exchange-vertex for each edge of $\textbf{B}$.
\end{corollary}

Since the set of bridges and edge-blocks induces a partition into the set of edges, we can define the following graph.

\begin{definition}
Given a graph $G$, the graph $G_\textbf{B}$ obtained from $G$ by contracting each edge-block into a vertex is called the \emph{edge-block} graph.
\end{definition}

\begin{proposition}
The edge-block graph $G_\textbf{B}$ of a connected graph $G$ is a tree.
\end{proposition}

We denote a vertex of $G_\textbf{B}$ as $[v]$ where $v$ is any vertex of the corresponding edge-block of $G$, i.e., $[v]$ is the equivalence class of vertices of an edge-block of $G$. If the equivalence class is trivial, we use $v$ instead of $[v]$.

For example, if $G$ is a unicyclic graph, the cycle of $G$ is denoted by $[v]$ in $G_\textbf{B}$ but the remainder vertices are denoted $u$ instead of $[u]$ in the edge-block graph of $G$.

\begin{proposition}
The edge $[u][v]$ of $G_\textbf{B}$ is a bridge if and only if $u_1v_1$ is a bridge of $G$ for some $u_1\in [u]$ and $v_1\in [v]$.
\end{proposition}

\begin{definition}
Let $G_\textbf{B}$ be the edge-block graph of $G$ and $f_t$ a $[k]$-connected configuration. We define a weight function $\omega[f_t] $ such that
\[\omega[f_t]:V(G_\textbf{B})\rightarrow \mathbb{N}\]
and $\omega[f_t]([v])$ is the number of empty vertices in the corresponding edge-block of $v$ in $G$.
\end{definition}

For example, if $f_t$ is a saturated connected configuration over $G$, the weighted function of $G_\textbf{B}$ is zero for any of its vertices. 

Let $[u][v]$ be an edge of $G_\textbf{B}$ and $u_1v_1$ the bridge such that $u_1\in [u]$ and $v_1\in [v]$. Recall that the set of empty vertices en the direction of $u_1$ with respect to $v_1$ is $b_{v_1}(u_1)$.  We denote by $\beta_{[v]}([u])$ to the sum of $\omega[f_t]([x])$ for all $[x]$ in the component of $G_\textbf{B}-[v]$ containing $[u]$.

\begin{lemma}
Let $[u][v]$ be an edge of $G_\textbf{B}$ where $[v]$ is not a trivial equivalence class. If $\beta_{[u]}([v])>0$ then $v_1$ is an exchange-vertex of $u_1v_1$ where $u_1v_1$ is the corresponding bridge for $[u][v]$.
\end{lemma}\label{lemma50}
\begin{proof}
Suppose $u_1,v_1\in V_t$. Since $[v]$ is not a trivial equivalence class, $v_1$ is a vertex of a cycle $C$. Let $v_2$ be another vertex of $C$. If $b_{v_1}(v_2)>0$ then there exists a valid movement $\sigma\in \Gamma[f_t]$ such that $\sigma(v_2)=v_1$ and $\sigma(v_1)=u_1$. Due to the fact that $v_1$ is an exchange-vertex of $v_1v_2$, $v_1$ is an exchange-vertex of $\sigma(v_1)\sigma(v_2)=u_1v_1$, see Figure \ref{fig15} (left).

\begin{figure}[!htbp]
\begin{center}
\includegraphics{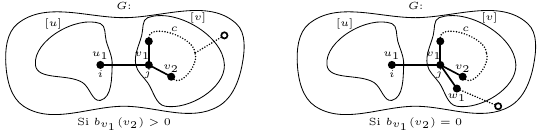}
\caption{If $[v]$ is not a trivial equivalence class, then $v_1$ is an exchange-vertex of $u_1v_1$.}\label{fig15}
\end{center}
\end{figure}

If $b_{v_1}(v_2)=0$, then $C$ is saturated and there is a valid movement $\sigma_1$ such that $\sigma_1(v_2)=v_1$. Since $u_1v_1$ is a bridge and $\beta_{[u]}([v])>0$ there exists a vertex $w_1\notin [v]$ adjacent to $v_1$ such that $b_{v_1}(w_1)>0$. Hence, there is a valid movement $\sigma_2\in \Gamma[f_t\circ \sigma_1]$ such that $\sigma_2(v_1)=u_1$ and $\sigma_2(v_1)$ and the result follows. See Figure \ref{fig15} (right).
\end{proof}

\begin{theorem}\label{teo51}
Let $P=(v_1,v_2,\dots,v_{r+1})$ be a path of $G$ such that each edge is a bridge and $v_1,v_{r+1}\in V_t$. If $[v_1]$ is a no trivial equivalence class of $G_\textbf{B}$ and $r\leq b_{v_2}(v_1)$, then $v_1$ is an exchange-vertex of each edge of $P$.
\end{theorem}
\begin{proof}
To begin with, observe that each vertex of $P$ has weight 1. For each $i\in [r]$ we have $b_{v_{i+1}}(v_i)\geq r$, then there exists a valid movement $\sigma_i$ such that $\sigma_i(v_1)=v_i$ and $\sigma_i(v_2)=v_{i+1}$. By Lemma \ref{lemma50}, $v_1$ is an exchange-vertex of $\sigma_i(v_1)\sigma_i(v_2)=v_iv_{i+1}$ because $b_{v_2}[f_t\circ\sigma_i]\geq r+1-i$, see Figure \ref{fig16}.

\begin{figure}[!htbp]
\begin{center}
\includegraphics{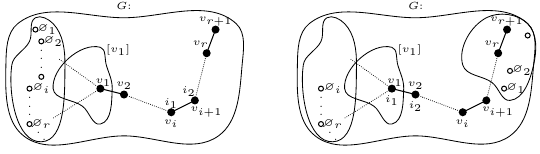}
\caption{If $[v]$ is not a trivial equivalence class, then $v_1$ is an exchange-vertex of $v_iv_{i+1}$.}\label{fig16}
\end{center}
\end{figure}
\end{proof}

\begin{corollary}\label{cor52}
Let $P=(v_1,v_2,\dots,v_{r+1})$ be a path of $G$ such that each edge is a bridge and $v_1,v_{r+1}\in V_t$. If $[v_1]$ is a no trivial equivalence class of $G_\textbf{B}$ and $r\leq b_{v_2}(v_1)$, then $v_1v_{r+1}\in\Phi[f_t]$.
\end{corollary}
\begin{proof}
By Theorem \ref{teo51}, $v_1$ is an exchange-vertex of each edge of $P$. By Theorem \ref{teo40}, the result follows.
\end{proof}

Next, we analyze the case where $[v]$ is a trivial equivalence class, but degree at least 3.

\begin{theorem}\label{teo53}
Let $P=(v_1,v_2,\dots,v_{r})$ be a path of $G$ such that each edge is a bridge and $v_1,v_{r}\in V_t$. If $[v_1]$ is a trivial equivalence class of $G_\textbf{B}$, degree at least 3 and $r\leq b_{v_2}(v_1)$, then $v_1$ is an exchange-vertex of each edge of $\{v_2,v_3,\dots,v_r\}$.
\end{theorem}
\begin{proof}
Observe that the condition $r\leq b_{v_2}(v_1)$ implies at least two empty vertices in the direction of $v_1$ with respect to $v_2$. Via a valid movement, we can obtain at least two empty vertices in at least one branches at $v_1$. Now, we use exactly the same argument as in Theorem \ref{teo51} and the result follows, see Figure \ref{fig17}.

\begin{figure}[!htbp]
\begin{center}
\includegraphics{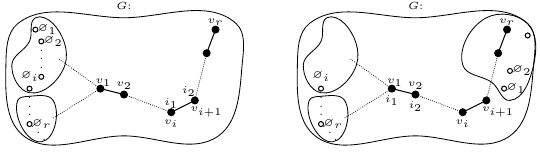}
\caption{If $[v]$ is a trivial equivalence class, then $v_1$ is an exchange-vertex of $v_iv_{i+1}$.}\label{fig17}
\end{center}
\end{figure}
\end{proof}

In the previous proof we remark if $f_t$ has at least two branches at $v_1$ with at least an empty vertex each, then $v_1$ is an exchange vertex of $(v_1v_2)$. See the vertices $[v_{17}]$ and $[v_7]$ of the example of Figure \ref{fig18}.

\begin{corollary}\label{cor54}
Let $P=(v_1,v_2,\dots,v_{r})$ be a path of $G$ such that each edge is a bridge and $v_1,v_{r}\in V_t$. If $[v_1]$ is a trivial equivalence class of $G_\textbf{B}$, degree at least 3 and $r\leq b_{v_2}(v_1)$, then $v_2v_{r}\in\Phi[f_t]$.
\end{corollary}

\begin{theorem}\label{teo55}
Let $[v]$ be a trivial equivalence class of $G_\textbf{B}$ with degree at least 3. If $x\in V_t$ (and $y\in V_t$) such that it is in direction of $v_1$ with respect to $v$ (and $v_2$ with respect to $v$, respectively) and at distance to $v$ less than $r$. If $b_{v_1}(v)\geq r$ (and $b_{v_2}(v)\geq r$) then $(xy)\in\Phi[f_t]$.
\end{theorem}
\begin{proof}
By Corollary \ref{cor54}, we have $(vv_1),(vv_2)\in\Phi[f_t]$. It remains to prove that $(v_1v_2)\in\Phi[f_t]$ owing to the fact that $(xy)=(v_1x)\circ(v_2y)\circ(v_1v_2)\circ(v_1x)\circ(v_2y)$.

If $vv_1$ and $vv_2$ satisfy the hypothesis of Theorem \ref{teo39}, the result follows. Without loss of generality, $vv_2$ doesn't have such hypothesis. Since $r\geq 2$, there exists a vertex $v_3$ adjacent to $v$  for which $b_v{v_3}\geq 2$. Hence, there is a valid movement $\sigma\Gamma[f_t]$ such that $\sigma(v_3)=v$ and $\sigma(v)=v_2$. Therefore, $b_v[f_t\circ\sigma](v_2)>0$ and $b_v[f_t\circ\sigma](v_3)>0$ and then $v$ is an exchange-vertex of $v_1$ and $v_2$ and finally $(v_1v_2)\in\Phi[f_t]$.
\end{proof}

Through Theorems \ref{teo51} and \ref{teo53} we can determine the vertices of $C_v$ where $[v]$ has degree at least 3 in $G_\textbf{B}$. We need the following useful definition related to $C_v$.

\begin{definition}
Let $[v]$ be a vertex of degree at least 3 of $G_\textbf{B}$.  We denote by $C_{[v]}$ to the subset of vertices of $G$ having the following property: for every pair of their vertices $(x,y)$, the transposition $(xy)$ is in $\Phi[f_t]$ according to Corollaries \ref{cor52} and \ref{cor54} and Theorems \ref{teo43} and \ref{teo55}.
\end{definition}

Figure \ref{fig18} shows the four sets $C_{[v]}$ of a graph $G$. Note that $b_{[v_4]}([v_3])=1$, so $C_{[v_3]}$ contains $v_2$, $v_3$ and $v_4$. Similarly $b_{[v_{10}]}([v_{11}])=2$, so $C_{[v_{11}]}$ contains $v_{9}$, $v_{10}$, $v_{11}$ and $v_{14}$. The vertices $[v_7]$ and $[v_{17}]$ of $G_\textbf{B}$ are trivial equivalence classes. For $C_{[v_7]}$ we have $b_{[v_6]}([v_7])=2$, $b_{[v_8]}([v_7])=1$ and $b_{[v_{15}]}([v_7])=3$ then $v_6\in C_{[v_7]}$ and $v_{15},v_{16}\in C_{[v_7]}$. Moreover, there are two branches at $v_7$ with empty vertices, then $v_7\in C_{[v_7]}$. Finally, $C_{[17]}$ contains $v_{18}$ and $v_{19}$ because $b_{[v_{16}]}([v_{17}])=0$, $b_{[v_{19}]}([v_{17}])=3$ and $b_{[v_{19}]}([v_{17}])=3$.

\begin{figure}[!htbp]
\begin{center}
\includegraphics{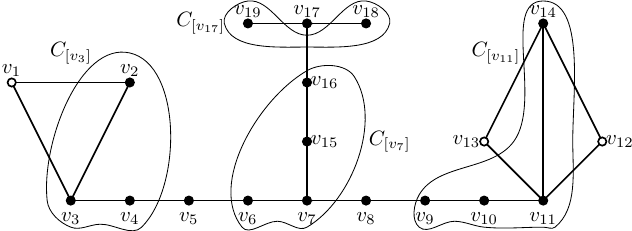}
\caption{The graph $G$ as an example with the set of vertices $C_{[v_3]}$, $C_{[v_7]}$, $C_{[v_{11}]}$ and $C_{[v_{17}]}$.}\label{fig18}
\end{center}
\end{figure}

Note the following facts. $S_{C_{[v]}}\subseteq \Phi[f_t]$ and if $v$ has degree at least 3, then $C_{[v]}\subseteq C_v$. Let $V_{\geq 3}$ be the set of vertices of degree at least 3.

We define the relation $R$ over $V_{\geq 3}$ as follows: 

$vRv$ if and only if there exists a sequence $\{u=u_1,u_2,\dots,u_{r+1}=v\}$ of vertices of $V_{\geq 3}$ such that for each $i\in [r]$, $C_{[u_i]}\cap C_{u_{i+1}}\not = \emptyset$. Clearly, $R$ is a equivalence relation. We denote by $R(v)$ the equivalence class of $v$.

\begin{theorem}
If $v\in V_{\geq 3}$, then \[C_v=\underset{u\in R(v)}{\bigcup}C_{[v]}.\]
\end{theorem}
\begin{proof}
Let $u\in R(v)$ be and a sequence $\{u=u_1,u_2,\dots,u_{r+1}=v\}$ of vertices of $V_{\geq 3}$ such that for each $i\in [r]$, $C_{[u_i]}\cap C_{u_{i+1}}\not = \emptyset$. Let $v_i\in C_{[u_i]}$
Let $v_0\in C_{[u]}$ then for each $j\in [r]$, we have $(v_{j-1}v_j)\in \Phi[f_t]$. Then, $(v_0v_r)\in \Phi[f_t]$. Since $v_r\in C_{[v]}\subseteq C_v$, then $v_0\in C_v$ and $C_{[u]}\subseteq C_v$.

Now, suppose that $v_0\in C_v$ and $v_0\in C_{[u]}$ for all $u\in R(v)$. Let $u_1\in R(u)$ such that $[u_1]$ is the closest vertex to $[v_0]$ in $G_\textbf{B}$. Then $C_{[u_1]}\cap C_{[v_0]}=\emptyset$. Since $v_0\in C_v$ there exist valid movements $\sigma$ and $\sigma_1$ such that $\sigma(v)=v_0$ and $\sigma_1(u_1)=v_0$ but it is not possible according to Theorems \ref{teo51} and \ref{teo53} because $v_0$ would have to be a vertex of $C_{u_1}$.
\end{proof}

Finally, note that if $x,y\in C_v$, then $(xy)\Phi[f_t]$ and then $S_{C_v}\subset \Phi[f_t]$. In consequence, we can prove the main theorem for no-saturated configurations.

\begin{theorem}
Let $R(v_1),R(v_2),\dots,R(v_m)$ be the equivalence classes of $R$. Then \[\Phi[f_t]=\underset{i=1}{\overset{m}{\prod}}S_{C_{v_{i}}}\times S_{V_\emptyset}.\]
\end{theorem}
\begin{proof}
By construction, $S_{V_\emptyset},S_{C_v}\leq \Phi[f_t]$, then $\underset{i=1}{\overset{m}{\prod}}S_{C_{v_{i}}}\times S_{V_\emptyset}\leq \Phi[f_t].$

Let $\sigma\in\Phi[f_t]$ be such that $\sigma\notin \underset{i=1}{\overset{m}{\prod}}S_{C_{v_{i}}}$, i.e., $\sigma(x)=y$ for $x\in C_{v_i}$ and $y\in C_{v_j}$ with $i\not = j$. The vertices $v_i$ and $v_j$ are exchange-vertices for some edges $xw$ and $yz$, respectively, then $(xw),(yz)\in\Phi[f_t]$.

On the other hand, if $\sigma(w)=z$ then $\sigma\in \Phi[f_t]$ and $v_i$ is an exchange-vertex for $xw$ in $f_t\circ\sigma$, then $v_i$ is an exchange-vertex for $yz$ in $f_t$. Hence, $y\in C_{v_i}$. A contradiction since $C_{v_i}\cap C_{v_j}=\emptyset$.

Now, suppose $\sigma(w)\not = z$, then $\sigma^{-1}(z)\not = w$. And we have $\sigma(x)=y$ then $\sigma^{-1}(z)\not =x$ and $\sigma(w)\not=y$. Therefore, $\sigma_1\colon=\sigma\circ(xw)\circ\sigma^{-1}\circ(yz)\circ\sigma\circ (xw)\in\Phi[f_t]$ and satisfies that $\sigma_1(x)=y$ and $\sigma_1(y)=z$. As before, this is a contradiction. Then $\sigma\in\underset{i=1}{\overset{m}{\prod}}S_{C_{v_{i}}}$ and $\Phi[f_t]\leq \underset{i=1}{\overset{m}{\prod}}S_{C_{v_{i}}}\times S_{V_\emptyset}$.
\end{proof}

To end this section, Figure \ref{fig18} shows a graph with a no-saturated configuration $f_t$ for which its Wilson group is $S_{C_{v_3}}\times S_{C_{v_7}}\times S_{C_{v_{11}}}\times S_{C_{v_{17}}} \times S_{V_\emptyset}$.

\section*{Acknowledgments}

Part of this work is included in the undergraduate thesis \cite{A} of one of the authors.

C.  Rubio-Montiel was partially supported by PAIDI grant 007/21.

The authors wish to thank the anonymous referees of this paper for their suggestions and remarks.

%---------------------------- Bibliography -------------------------------

\end{document}